\documentclass{amsart}
\usepackage{tikz}
\usepackage{xcolor}
\usepackage{enumerate}
\usepackage{amssymb}
\date{\today}
\title[$t$-Cores for $(\Delta+t)$-edge-colouring]{$t$-Cores for $(\Delta+t)$-edge-colouring}
\author{Jessica McDonald}
\author{Gregory J.~Puleo}
\thanks{The first author is supported in part by NSF grant DMS-1600551}
\address{Department of Mathematics and Statistics, Auburn University, Auburn, Alabama, USA 36849}
\email[Jessica McDonald]{mcdonald@auburn.edu}
\email[Gregory J. Puleo]{gjp0007@auburn.edu}
\newcommand{\greg}[1]{\textcolor{green!50!black}{#1}}
\renewcommand{\phi}{\varphi}

\tikzstyle{vertex}=[inner sep = 0pt, minimum width=4pt, fill=black, shape=circle]
\tikzstyle{squarevert}=[inner sep = 0pt, minimum width=4pt, minimum height=4pt, fill=white, shape=rectangle, draw=black, thick]
\newcommand{\gpoint}[2]{\node[style=vertex, label=#1:$#2$]}

\newcommand{\bpoint}[1]{\gpoint{below}{#1}}
\newcommand{\apoint}[1]{\gpoint{above}{#1}}

\newcommand{\rpoint}[1]{\gpoint{right}{#1}}

\renewcommand{\subset}{\subseteq}

\newcommand{\caze}[2]{\textbf{Case {#1}:} \textit{#2}}
\newcommand{\stage}[1]{\textbf{Stage {#1}.}}
\newcommand{\sizeof}[1]{\left\lvert{#1}\right\rvert}

\newcommand{\st}{\colon\,}
\newtheorem{proposition}{Proposition}[section]

\newtheorem{theorem}[proposition]{Theorem}
\newtheorem{lemma}[proposition]{Lemma}

\theoremstyle{definition}

\theoremstyle{remark}
\newtheorem{remark}[proposition]{Remark}

\DeclareMathOperator{\cdeg}{cdeg}
\DeclareMathOperator{\Fan}{Fan}
\DeclareMathOperator{\corefan}{cfan}
\DeclareMathOperator{\fan}{fan}
\begin{document}
\maketitle
\begin{abstract}
  We extend the edge-colouring notion of \emph{core} (subgraph induced
  by the vertices of maximum degree) to \emph{$t$-core} (subgraph
  induced by the vertices $v$ with $d(v)+\mu(v)> \Delta+t$), and find
  a sufficient condition for $(\Delta+t)$-edge-colouring. In
  particular, we show that for any $t\geq 0$, if the $t$-core of $G$
  has multiplicity at most $t+1$, with its edges of multiplicity $t+1$
  inducing a multiforest, then $\chi'(G) \leq \Delta+t$. This extends
  previous work of Ore, Fournier, and Berge and Fournier. A stronger
  version of our result (which replaces the multiforest condition with
  a vertex-ordering condition) generalizes a theorem of Hoffman and
  Rodger about cores of $\Delta$-edge-colourable simple graphs. In
  fact, our bounds hold not only for chromatic index, but for the
  \emph{fan number} of a graph, a parameter introduced by Scheide and
  Stiebitz as an upper bound on chromatic index.  We are able to give
  an exact characterization of the graphs $H$ such that
  $\Fan(G) \leq \Delta(G)+t$ whenever $G$ has $H$ as its $t$-core.
\end{abstract}
\section{Introduction}\label{sec:introduction}

In this paper a \emph{graph} is permitted to have parallel edges but no loops; we will say \emph{simple graph} when we wish to disallow parallel edges.

A \emph{$k$-edge-colouring} of a graph $G$ is a function that assigns a
colour from $\{1, \ldots, k\}$ to each edge of $G$ so that adjacent
edges receive different colours. The \emph{chromatic index} of $G$,
$\chi'(G)$, is the minimum $k$ such that $G$ is $k$-edge-colourable;
the maximum degree $\Delta(G)$ is an obvious lower bound for
$\chi'(G)$.  When the graph $G$ is understood, we sometimes write $\Delta$ for
$\Delta(G)$.

Numerous authors have found sufficient conditions for
$\Delta$-edge-colouring a simple graph $G$ by studying its \emph{core},
that is, the graph induced by the its vertices of degree $\Delta$.
An early such result is due to Fournier~\cite{fournier77, fournier73}:
\begin{theorem}[Fournier~\cite{fournier77, fournier73}]\label{thm:fournier}
  If $G$ is a simple graph and the core of $G$ is a forest, then
  $\chi'(G) = \Delta(G)$.
\end{theorem}
This result was strengthened by Hoffman and Rodger~\cite{hoffman-rodger} who showed that if $B$ is the core of a graph $G$, and $B$ permits a specific vertex-ordering called a \emph{full B-queue}, then $G$ is $\Delta$-edge-colourable. We defer a precise definition of \emph{full B-queue} to Section~\ref{sec:bqueue} of this paper, but we state their result now, noting that if $B$ is a forest, then it indeed has a full B-queue.  Hoffman and Rodger \cite{hoffman-rodger} also provided an efficient algorithm for deciding whether or not a graph $B$ has a full $B$-queue; in fact they showed that the greedy algorithm works.

\begin{theorem}[Hoffman--Rodger~\cite{hoffman-rodger}]\label{thm:hoffman-rodger}
  Let $G$ be a simple graph with core $B$. If $B$ has a
  full $B$-queue, then $\chi'(G) = \Delta(G)$.
\end{theorem}

Simple graphs can be divided into those of \emph{class I} (having chromatic index $\Delta$) or \emph{class II} (having chromatic index $\Delta+1$), but in general the chromatic index of $G$ can be as high as $\Delta+\mu$, where $\mu=\mu(G)$ is the maximum edge-multiplicity of $G$. This classical bound of Vizing \cite{vizing} also has the following local refinement due to Ore, where $\mu(v)$ denotes the maximum edge multiplicity incident to vertex $v$.
\begin{theorem}[Ore~\cite{ore-fourcolor}]\label{thm:ore}
  For every graph $G$, $\chi'(G) \leq \max_{v \in V(G)}[d(v) + \mu(v)]$.
\end{theorem}
We define the \emph{t-core} of $G$ to be the subgraph induced by the vertices $v$ with \[d(v)+\mu(v)> \Delta+t.\] Observe that the $0$-core of a nonempty
simple graph is simply its core. Ore's Theorem can be restated as: ``For any $t\geq 0$, if  the $t$-core of a graph $G$ is empty, then $G$ is   $(\Delta+t)$-edge-colourable''. We improve this and generalize
Theorem~\ref{thm:fournier} as follows. Here, by \emph{multiforest},
we mean a graph whose underlying simple graph is a forest.
\begin{theorem}\label{thm:forestcore}
  Let $G$ be a graph and let $t\geq 0$. If the $t$-core of $G$ has multiplicity at most $t+1$, with its edges of multiplicity $t+1$ inducing a multiforest, then $\chi'(G) \leq \Delta+t$.
\end{theorem}
The $t=0$ case of Theorem~\ref{thm:forestcore} implies
Theorem~\ref{thm:fournier} (and is already slightly stronger, since
Theorem~\ref{thm:forestcore} allows $G$ to be a multigraph even though
$t=0$ forces the $0$-core of $G$ to be simple whenever the hypothesis
is met). When $t=\mu(G)-1$, the hypothesis of
Theorem~\ref{thm:forestcore} is just that the edges of multiplicity
$\mu$ in the $t$-core induce a multiforest; this strengthens a
previous result of Berge and Fournier \cite{berge-fournier}, who
showed that if the $(\mu-1)$-core of $G$ is edgeless, then $G$ is
$(\Delta+\mu-1)$-edge-colourable.

The multiplicity condition in Theorem~\ref{thm:forestcore} is sharp,
and this can already be seen with a fat triangle. Consider the
multigraph $G$ obtained from $K_3$ by giving two edges multiplicity
$t+1$ and the remaining edge multiplicity $t+2$. Now
$\Delta(G) = 2t+3$, and the $t$-core of $G$ is simply the $t+2$
parallel edges (since for each of those endpoints, degree plus
multiplicity is $3t+5>\Delta(G)+t$, while for the other vertex this
sum is only $3t+3$). Hence, the $t$-core of $G$ is a multiforest but
with multiplicity $t+2$; this discrepancy from Theorem
\ref{thm:forestcore} is already enough to cause a problem, as of
course this fat triangle has $\chi'(G)=3t+4 > \Delta(G)+t$.

Theorem \ref{thm:forestcore} is in fact a corollary of a stronger result we prove, which generalizes Theorem \ref{thm:hoffman-rodger}. Theorem \ref{thm:hoffman-rodger} is about a condition on the core (0-core) of a simple graph that guarantees $\Delta$-edge-colourability; here we get a condition on the $t$-core of a graph that guarantees $(\Delta+t)$-edge-colourability (with the same condition when $t=0$).

\begin{theorem}\label{thm:Bqueuecore} Let $G$ be a graph, let $t\geq 0$, and let $H$ be the $t$-core of $G$. If $H$ has multiplicity at most $t+1$, and the underlying simple graph $B$ of those maximum multiplicity edges has a full $B$-queue, then $\chi'(G)\leq \Delta(G)+t$.
\end{theorem}

We can actually state Theorem \ref{thm:Bqueuecore} (and hence Theorem \ref{thm:forestcore}) in an even stronger way, by replacing $\chi'(G)$ with the \emph{fan number} $\Fan(G)$.  Scheide and Stiebitz~\cite{SS} introduced $\Fan(G)$ to essentially describe the smallest $k$ for which
Vizing's Fan Inequality (see Section~\ref{sec:fan}) can be used to prove that $G$ is
$k$-edge-colourable, in particular proving the following.

\begin{theorem}[Scheide--Stiebitz~\cite{SS}]\label{thm:SS}
  For any graph $G$, $\chi'(G) \leq  \Fan(G)$.
\end{theorem}

We are able to give an exact characterization of the graphs $H$  such that $\Fan(G) \leq \Delta(G)+t$ whenever $G$ has $H$ as its $t$-core. In particular, we will define $\corefan(H)$
for a graph $H$ (which we'll think of as being the $t$-core of
$G$), and prove the following pair of theorems.

\begin{theorem}\label{thm:corefan}
  Let $G$ be a graph, let $t\geq 0$, and let $H$ be
  the $t$-core of $G$. If $\corefan(H) \leq t$, then $\Fan(G)\leq \Delta+t$.
\end{theorem}

\begin{theorem}\label{thm:converse}
  Let $H$ be a graph, and let $t$ be a nonnegative integer. If
  $\corefan(H) > t$, then there exists a graph $G$ with
  $t$-core $H$ such that $\Fan(G)> \Delta(G)+t$.
\end{theorem}

This pair of results can be thought of as a sort of multigraph analog
to the work of Hoffman~\cite{hoffman}, who found a necessary and
sufficient condition for a simple graph $H$ to be the core of a simple
graph $G$ containing a so-called overfull subgraph of the same maximum
degree. Overfull graphs are known to be class II. The graph $G$
constructed in Theorem~\ref{thm:converse} does not necessarily satisfy
$\chi'(G) > \Delta(G) + t$, as one might hope, but the lower bound on
the fan number suggests that fan-recolouring would not suffice to
$(\Delta+t)$-edge-colour these graphs.

Our paper is organized as follows. We'll define $\Fan$ and $\corefan$
in Section~\ref{sec:fan}, spending time to motivate these definitions
according to Vizing's Adjacency Lemma, and conclude the section with a
proof of Theorem \ref{thm:corefan}. In Section~\ref{sec:bqueue} we'll
give a precise definition of $B$-queue and full $B$-queue, and prove
Theorem \ref{thm:Bqueuecore}. In particular, we'll show that when $H$
is the $t$-core of $G$, and $H$ has all the assumptions of Theorem
\ref{thm:Bqueuecore}, then $\corefan(H)\leq t$, and hence Theorems
\ref{thm:SS} and \ref{thm:corefan} imply that
$\chi'(G)\leq Fan(G)\leq \Delta+t$. Our proof of Theorem
\ref{thm:converse} is the subject of Section~\ref{sec:converse}.

\begin{remark}
  The word ``core'' has several different meanings in graph theory. In
  addition to the usage above, it has a definition in the setting of
  graph homomorphisms. Moreover, the term ``$k$-core'' has also been
  used in a degeneracy context, to refer to the component of $G$ that
  remains after iteratively deleting vertices of degree at most $k$.
\end{remark}

\section{Proof of Theorem \ref{thm:corefan}}\label{sec:fan}

In the introduction we described $\Fan(G)$ as essentially describe the
smallest $k$ for which the following theorem, Vizing's Fan Inequality, can be used to
prove that $G$ is $k$-edge-colourable. Let us now say more about this.
prove that $G$ is $k$-edge-colourable. Let us now say more about this.

\begin{theorem}\label{thm:vizfan}\emph{(Vizing's Fan Inequality
    \cite{vizing}, see also \cite{SSFT})} Let $G$ be a graph, let
  $k\geq \Delta$, and suppose there is a $k$-edge-colouring of $J-e$
  for some $J\subseteq G$ and $e=xy \in E(G)$. Then either $J$ is
  $k$-edge-colourable, or there exists a vertex-set
  $Z\subseteq N_J(x)$ such that $|Z|\geq 2$, $y\in Z$, and
\begin{equation}\label{fanineq}
\sum_{z\in Z} \left( d_J(x) +\mu_J(x, z) - k\right) \geq 2.
\end{equation}
\end{theorem}

Vizing's Theorem (and Ore's Theorem) follow immediately from the fan
inequality. To see this, consider an edge-minimal counterexample $G$
(so let $J=G$ in Theorem \ref{thm:vizfan}), and note that setting
$k=\Delta(G)+\mu(G)$ (or $k=\max_{v \in V(G)}[d(v) + \mu(v)]$) makes
inequality (\ref{fanineq}) impossible to satisfy.

In order to apply Theorem \ref{thm:vizfan}, we would certainly need
$k\geq \Delta$. Given this however, if we had a $k$-edge-colouring of
$J-e$ for some $e=xy\in E(J)$ and we knew that for \emph{every}
$Z\subseteq N(x)$ with $y\in Z$ and $\sizeof{Z} \geq 2$,
\[\sum_{z \in Z}(d_J(z) + \mu_J(x,z) - k) \leq 1,\]
then we'd get a proof of $k$-edge-colourability of $J$ via Theorem
\ref{thm:vizfan}. On the other hand, if we knew that
\[d_J(x) + d_J(y) - \mu_J(x,y) \leq k,\] for such an $e=xy$, then we'd
get our $k$-edge-colouring extending to $J$ simply because $e$ sees at
most $k-1$ different edges in $G$. With this in mind, Scheide and
Stiebitz~\cite{SS} defined the \emph{fan-degree}, $\deg_J(x,y)$, of
the pair $x, y\in V(J)$ as the smallest nonnegative integer $k$ such
that either:
\begin{enumerate}[(i)]
\item $d_J(x) + d_J(y) - \mu_J(x,y) \leq k$, or
\item $\sum_{z \in Z}(d_J(z) + \mu_J(x,z) - k) \leq 1$ for all $Z \subset N_J(x)$ with $y \in Z$ and $\sizeof{Z} \geq 2$.
\end{enumerate}
So, we could extend the $k$-edge-colouring of $J-e$ to $J$ provided we
knew that $\deg_J(x,y)\leq k$. Of course, our goal is to
$k$-edge-colour all of $G$, not just some subgraph $J$. However, if
$G$ is not $k$-edge-colourable, then there exists $J\subseteq G$ with
the property that $J-e$ is $k$-edge-colourable for all $e\in E(J)$ but
$J$ is not $k$-edge-colourable. If, for \emph{this} $J$, we knew that
there was a choice of $xy\in E(J)$ with $d_J(x,y)\leq k$, then we'd
know that $J$ is $k$-edge-colourable after all, and hence so is
$G$. If such a choice of $xy$ existed for \emph{every} subgraph $J$ of
$G$ (say with at least one edge), then we would certainly get that $G$
is $k$-edge-colourable. Hence, Scheide and Stiebitz~\cite{SS} defined
the \emph{fan number}, $\fan(G)$, of a graph $G$ by
\[ \fan(G) = \max_{J \subset G, E(J)\neq \emptyset} \min\{\deg_J(x,y) \st xy \in E(J)\}, \]
with $\fan(G)$ defined to be 0 for an edgeless graph $G$. Recalling the requirement that $k\geq \Delta$, they finally defined $\Fan(G)=\max\{\Delta, \fan(G)\}$, and established Theorem \ref{thm:SS}.

Now suppose that the graph $G$ has $t$-core $H$. We would like to be
able to look just at $H$ and determine that $\Fan(G)\leq \Delta+
t$. To this end, we would like to describe a condition on $H$ that
would guarantee that for every $J\subseteq G$, there exists
$xy\in E(J)$ with $\deg_J(x,y)\leq \Delta+t$. We'll forget about (i)
for this purpose, and try to get a condition on $H$ which guarantees
(ii) for such $J, x, y$. If $K=J\cap H$, then we're trying to get a
guarantee for $J$ by only looking at $K$. The good news here is that
if, for example, some vertex $z\in Z$ is in $J$ but not $K$, then $z$
is not in the $t$-core, so in particular,
\[d_J(z) + \mu_J(x,z) - (\Delta(G)+t)\leq 0,\]
that is, the vertex $z$ is insignificant in terms of establishing (ii). There are more details to handle, but we'll see that the following definition is the right condition to require. Note that while we'll think of $H$ as being the $t$-core of a graph $G$, this definition takes as input any graph $H$.

For any graph $H$, subgraph $K \subset H$, and ordered pair of
  vertices $(x,y)$ with $xy \in E(K)$, we define the \emph{cfan
    degree}, denoted $\cdeg_{H,K}(x,y)$, as the smallest nonnegative integer $l$
  such that for all $Z \subset N_K(x)$ with $y \in Z$, we have
  \[ \sum_{z \in Z}(d_K(z) - d_H(z) + \mu_K(x,z) - l) \leq 1. \]
  Note that, in contrast to the fan degree, the cfan degree does
    \emph{not} impose the restriction that $\sizeof{Z} \geq 2$ when
    determining which sets $Z \subset N_K(x)$ must be considered.
  The \emph{cfan number} of $H$, written $\corefan(H)$, is then
  defined by
  \[ \corefan(H) = \max_{K \subset H, E(K)\neq\emptyset}\min\{ \cdeg_{H,K}(x,y) \st xy \in E(K) \}, \]
  with $\corefan(H)$ defined to be 0 for an edgeless graph $H$. With this definition established, we can now prove Theorem \ref{thm:corefan}.

\begin{proof}[Proof of Theorem~\ref{thm:corefan}]
  Suppose that $\corefan(H) \leq t$. We will show that this implies that $\fan(G)\leq \Delta+t$, which in turn implies that $\Fan(G)\leq \Delta+t$, as desired.

  If $G$ is an edgeless graph, then $\fan(G)=0$ by definition, so our result is immediate. Now suppose that $G$ has at least one edge, and let any subgraph $J \subset G$ with $E(J)\neq \emptyset$ be given. We will show that there exists $xy\in E(G)$ with
  $\deg_J(x,y)\leq \Delta(G)+t$; in particular we will show that for all $Z \subset N_J(x)$ with $y \in Z$ and $\sizeof{Z} \geq 2$,
  \[\sum_{z \in Z}(d_J(z) + \mu_J(x,z) - (\Delta(G)+t)) \leq 1.\]

  Let $K = J \cap H$. We consider two cases:
  either $K$ contains an edge, or $K$ contains no edges.

  \caze{1}{$K$ contains an edge.} In this case, since $K\subseteq H$ and we know that $\corefan(H)\leq t$, we know that there exists $xy\in E(K)$ with $\cdeg_{H, K}(x,y)\leq t$, that is, with
  \[\sum_{z \in Z}(d_K(z) - d_H(z) + \mu_K(x,z) - t) \leq 1 \]
  for all $Z \subset N_K(x)$ with $y \in Z$. Note that for any $z\in V(K)$,
  \[d_J(z)-d_K(z)+d_H(z)\leq d_G(z)\leq \Delta(G).\]
  So we get that for all $Z \subset N_K(x)$ with $y \in Z$,
  \[\sum_{z \in Z}(d_J(z) + \mu_K(x,z) - (\Delta(G)+t)) \leq 1\]
  Now observe that if $w \in N_J(x) - V(H)$, then by the definition of
  the $t$-core of $G$, we have $d_G(w) + \mu(w) \leq \Delta(G)+t$. So
  in fact we can say that the above sum holds for all
  $Z\subseteq N_J(z)$ with $y\in Z$ and $\sizeof{Z} \geq 2$, as
  desired. (Note that this is the reason we cannot impose the
    restriction that $\sizeof{Z} \geq 2$ in the definition of cfan
    degree: if we imposed that restriction and had $N_K(x) = \{y\}$ but
    $\sizeof{N_J(x)} \geq 2$, we would have no control over the value
    of $d_J(y) + \mu_K(x,y) - (\Delta(G)+t)$.)

  \caze{2}{$K$ has no edges.} In this case, let $(x,y)$ be any pair such that $xy \in E(J)$,
  taking $x \in V(H)$ if possible. Our choice of $x$ implies that for all $z \in N_J(x)$,
  we have $z \notin V(H)$, hence $d_G(z) + \mu_G(z) \leq \Delta(G)+t$ by the definition of a $t$-core.
  Thus, for every $Z \subset N_J(x)$ with $y \in Z$ and $\sizeof{Z} \geq 2$, we have
  \[ \sum_{z \in Z}(d_J(z) + \mu_J(x,z) - (\Delta(G)+t)) \leq \sum_{z \in
    Z}(d_G(z) + \mu_G(x,z) - (\Delta(G)+t)) \leq 1, \]
  as needed.
\end{proof}

\section{Proof of Theorem \ref{thm:Bqueuecore}}\label{sec:bqueue}

We start this section by providing the definition of a \emph{full B-queue}, which is needed for Theorem \ref{thm:Bqueuecore} (and for Theorem \ref{thm:hoffman-rodger}). Hoffman and Rodger \cite{hoffman-rodger} defined a \emph{$B$-queue} of a simple graph
$B$ to be a sequence of vertices $(u_1, \ldots, u_q)$ and a sequence
of vertex subsets $(S_0, S_1, \ldots, S_q)$ such that:
  \begin{enumerate}[(i)]
  \item $S_0 = \emptyset$, and
  \item For all $i \in [q]$:
    \begin{itemize}
    \item $S_i = N(u_i) \cup \{u_i\} \cup S_{i-1}$,
    \item $1 \leq \sizeof{S_i \setminus S_{i-1}} \leq 2$,
    \item $u_i \notin \{u_1, \ldots, u_{i-1}\}$, and
    \item $\sizeof{S_i \setminus (S_{i-1} \cup \{u_i\})} \leq 1$.
    \end{itemize}
  \end{enumerate}
  If $S_q = V(B)$ then we say the $B$-queue is \emph{full}. We noted
  in the introduction that every simple forest $B$ admits a full
  $B$-queue. To see this, first suppose that a $B$-queue
  $(u_1, \ldots, u_{i-1})$ and $(S_0, \ldots, S_{i-1})$ has already
  been defined for $B$, but the $B$-queue is not full,
  ie. $S_{i-1}\neq V(B)$. If $B-S_{i-1}$ consists only of isolated
  vertices, then they may be chosen in any order as
  $u_{i}, u_{i+1}, \ldots$ so as to get a full $B$-queue. If not, then
  $B-S_{i-1}$ is a forest, so it contains a leaf vertex which can be
  chosen for $u_{i}$. With this choice $|S_{i}-S_{i-1}|=2$ and
  $\sizeof{S_{i} \setminus (S_{i} \cup \{u_i\})}=1$ (since
  $u_{i}\not\in S_{i-1}$ in this case), so $(u_1, \ldots, u_i)$ and
  $(S_0, \ldots, S_i)$ is again a $B$-queue. This process can be
  repeated until the $B$-queue is full.

In addition to forests, there are many other simple graphs $B$ that have full $B$-queues. For example, while a cycle itself does not have a full $B$-queue (there is no valid choice for $u_1$), adding any number of pendant edges to a cycle allows the same procedure described above to yield a full $B$-queue. For another, more complicated example, see \cite{hoffman-rodger}.

In this section, we'll prove the following result.

\begin{theorem}\label{thm:Bqueuecorefan} Let $G$ be a graph, let $t\geq 0$, and let $H$ be the $t$-core of $G$. If $H$ has multiplicity at most $t+1$, and the underlying simple graph $B$ of those maximum multiplicity edges has a full $B$-queue, then $\corefan(H)\leq t$.
\end{theorem}

Given the conclusion of Theorem \ref{thm:Bqueuecorefan}, Theorems \ref{thm:SS} and \ref{thm:corefan} immediately tell us that
\[ \chi'(G)\leq \Fan(G)\leq \Delta+t, \]
which in particular implies Theorem \ref{thm:Bqueuecore}.

We'll prove Theorem \ref{thm:Bqueuecorefan} by establishing a sequence of lesser results. The first such lemma, which follows, says that when looking for an upper bound on $\corefan(H)$, it suffices to look at a subgraph of $H$ formed by high-multiplicity edges.

\begin{lemma}\label{lem:lowmult}
  Let $H$ be a graph, let $t$ be a nonnegative integer, and let $H_{>t}$ be
  the subgraph of $H$ consisting of the edges with multiplicity greater
  than $t$. The following are equivalent:
  \begin{enumerate}[(i)]
  \item $\corefan(H) \leq t$,
  \item $\corefan(H_{>t}) \leq t$.
  \end{enumerate}
\end{lemma}

\begin{proof}
  Let $H' = H_{>t}$.
  All nonempty subgraphs of $H'$ are also subgraphs of $H$ which must be considered
  when computing $\corefan(H)$, so (i)$\implies$(ii) is immediate. To show that (ii)$\implies$(i),
  let $K$ be any nonempty subgraph of $H$. We will find a pair $(x,y)$ with $\cdeg_{H,K}(x,y) \leq t$.

  If all edges of $K$ have multiplicity at most $t$, then let $(x,y)$
  be any pair with $xy \in E(K)$.  For any $Z \subset N_K(x)$ with
  $y \in Z$, all terms of the sum
  \[ \sum_{z \in Z}(d_K(z) - d_H(z) + \mu_K(x,z) - t) \] are
  nonpositive, so this sum is clearly at most $1$, as desired.

  Thus, we may assume that $K$ has some edges of multiplicity at least
  $t+1$. Let $K' = K \cap E(H')$; now $K'$ is a nonempty subgraph
  of $H'$, so we obtain a pair $(x,y)$ such that
  $\cdeg_{H', K'}(x,y) \leq t$.  We claim that also
  $\cdeg_{H, K}(x,y) \leq t$.

  For any $Z \subset N_K(x)$ with
  $y \in Z$, let $Z' = Z \cap N_{K'}(x)$.  For any $z \in Z-Z'$, we
  have $\mu_K(x,z) \leq t$, so the contribution of $z$ to the sum
  \[ \sum_{z \in Z}[d_K(z) - d_H(z) + \mu_K(x,z) - t] \]
  is nonpositive. Moreover, for every $v \in V(K)$, every edge
    that is lost when we pass from $K$ to $K'$ is also lost when we
    pass from $H$ to $H'$, so that $d_K(v) - d_{K'}(v) \leq d_H(v) - d_{H'}(v)$,
    which rearranges to $d_K(v) - d_{H}(v) \leq d_{K'}(v) - d_{H'}(v)$. Since
    also $\mu_K(u,v) = \mu_{K'}(u,v)$ for all $uv \in E(K')$, this yields
    \begin{align*}
      \sum_{z \in Z}[d_K(z) - d_H(z) + \mu_K(x,z) - t] &\leq \sum_{z \in Z'}[d_K(z) - d_H(z) + \mu_K(x,z) - t] \\
      &\leq \sum_{z \in Z'}[d_{K'}(z) - d_{H'}(z) + \mu_{K'}(x,z) - t] \leq 1,
    \end{align*}
  where the last inequality follows from $\cdeg_{H', K'}(x,y) \leq t$.
\end{proof}

When trying to determine $\corefan$ for a given graph, one need only
focus on \emph{full multiplicity} subgraphs, as indicated in the
following lemma, and we'll see this to be a helpful idea. A subgraph
$K$ of a graph $H$ has \emph{full multiplicity} if
$\mu_{K}(e)=\mu_H(e)$ for all $e\in E(K)$. (Note that some edges of
$H$ may be omitted from $K$ entirely.)

\begin{lemma}\label{lem:fullmult}
    For any graph $H$,
    \[ \corefan(H) = \max_K \min \{ \cdeg_{H,K}(x,y) \st xy \in E(K) \}, \]
    where the maximum is taken over all nonempty subgraphs $K \subset H$ such that $K$ has full
    multiplicity.
  \end{lemma}

\begin{proof} Consider a subgraph $K$ of graph $H$, and define $K'$ to be the subgraph of $H$ having the same underlying simple graph as $K$, but having full multiplicity. Note that the truth of the lemma would follow if we could establish
\[ \cdeg_{H,K'}(x,y) \geq \cdeg_{H,K}(x,y) \]
  for all pairs $(x,y)$ with $xy \in E(K)$. To this end, note that the definition of cdeg gives us
    \begin{equation}
      \label{eq:znkprime}
      \sum_{z \in Z'}(d_{K'}(z) - d_H(z) + \mu_{K'}(x,y) - \cdeg_{H,K'}(x,y)) \leq 1
    \end{equation}
    for all $Z' \subset N_{K'}(x)$ with $y \in Z$.
    Since $d_{K'}(z) \geq d_{K}(z)$ and $\mu_{K'}(x,y) \geq \mu_K(x,y)$ for all $x,y$, we see that
    for all such $Z'$, we also have
    \begin{equation}
      \label{eq:znk}
      \sum_{z \in Z'}(d_{K}(z) - d_{H}(z) + \mu_{K}(x,y) - \cdeg_{H,K'}(x,y)) \leq 1.
    \end{equation}
    Since $N_{K'}(x) = N_{K}(x)$ for all $x$, we see that Inequality~\eqref{eq:znk} holds
    (with $Z$ replacing $Z'$) for all $Z \subset N_{K}(x)$ with $y \in Z$. This gives us our desired inequality.
\end{proof}

We now turn our focus to computing $\corefan$ in graphs of
\emph{constant multiplicity}, that is, graphs where every edge has the
same multiplicity.

\begin{lemma}\label{lem:tplus1}
  Let $H$ be a graph of constant multiplicity $t+1$, and for every
  nonempty subgraph $K \subset H$, let $Z(K) = \{v \in V(K) \st d_K(v) = d_H(v)\}$.
  The following are equivalent:
  \begin{enumerate}[(i)]
  \item $\corefan(H) \leq t$,
  \item For every nonempty full-multiplicity subgraph $K \subset H$, there is an edge $xy \in E(K)$
    such that $\sizeof{(N_H(x) \cap Z(K)) - y} \leq d_H(y) - d_K(y)$.
  \end{enumerate}
\end{lemma}
\begin{proof}
  (i)$\implies$(ii): Let any nonempty full-multiplicity subgraph $K \subset H$ be given,
  and let $(x,y)$ be a pair such that $xy \in E(K)$ and $\cdeg_{H,K}(x,y) \leq t$. Let $Z = (N_H(x) \cap Z(K)) \cup \{y\}$.
  Observe that $z \in N_H(x) \cap Z(K)$ implies that $xz \in E(K)$, so that $Z \subset N_K(x)$;
  thus, since $\cdeg_{H,K}(x,y) \leq t$, we have
  \[ \sum_{z \in Z}[ d_K(z) - d_H(z) + \mu_K(x,z) - t] \leq 1. \]
  Since vertices in $Z - y$ contribute exactly $1$ to this sum, this implies that
  \[ \sizeof{Z - y} + (d_K(y) - d_H(y) + \mu_K(x,y) - t) \leq 1, \]
  so by the definition of $Z$,
  \[ \sizeof{(N_H(x) \cap Z(k)) -y} \leq d_H(y) - d_K(y) - \mu_K(x,y) + t + 1 = d_H(y) - d_K(y), \]
  where in the last equation we have $\mu_K(x,y) = t+1$ since $K$ is full-multiplicity and $xy \in E(K)$.
  \smallskip

  (ii)$\implies$(i): We apply Lemma~\ref{lem:fullmult}. Let any
  nonempty full-multiplicity subgraph $K \subset H$ be given, and let
  $xy$ be an edge such that
  $\sizeof{(N_H(x) \cap Z(K)) - y} \leq d_H(y) - d_K(y)$.  We claim
  that $\cdeg_{H,K}(x,y) \leq t$.

  Let $Z$ be any subset of $N_K(x)$ containing $y$. We must show that
  $\sum_{z \in Z}[ d_K(z) - d_H(z) + \mu_K(x,z) - t)] \leq 1$. Observe
  that any elements of $Z - Z(K)$ contribute
  a nonpositive term to this sum, while elements of $Z(K)$ contribute $1$, so that
  \begin{align*}
    \sum_{z \in Z}[ d_K(z) - d_H(z) + \mu_K(x,z) - t] &\leq \sum_{z \in (N_H(x)\cap Z(K)) \cup \{y\}}[ d_K(z) - d_H(z) + \mu_K(x,z) - t ] \\
    &\leq \sizeof{(N_H(x) \cap Z(K)) - y} + (d_K(y) - d_H(y) + 1) \\
    &\leq (d_H(y) - d_K(y)) + (d_K(y) - d_H(y) + 1) \\
    &= 1.\qedhere
  \end{align*}
\end{proof}

\begin{lemma}\label{lem:flatten}
  Let $B$ be a simple graph, and let $B_s$ and $B_t$ be graphs of
  constant multiplicity $s+1$ and $t+1$ respectively, with underlying
  simple graph $B$. If $0\leq s < t$ and $\corefan(B_s) \leq s$, then
  $\corefan(B_t) \leq t$.
\end{lemma}
\begin{proof}
  First observe that for every vertex $v \in V(B)$, we have
  $N_B(v) = N_{B_s}(v) = N_{B_t}(v)$; thus, we suppress the
  subscripts and simply write $N(v)$. To avoid double-subscripts, we
  will also write $d_s$ and $\mu_s$ as shorthand for $d_{B_s}$ and
  $\mu_{B_s}$, and likewise for $t$.

  We verify Condition~(ii) of Lemma~\ref{lem:tplus1} for $B_t$. Let
  $K$ be any nonempty full-multiplicity subgraph of $B_t$, and let
  $K'$ be the full-multiplicity subgraph of $B_s$ having the same
  underlying simple graph. Observe that $Z(K') = Z(K)$, and that $K'$
  is nonempty since it has the same underlying simple graph as $K$.

  Applying Lemma~\ref{lem:tplus1} to $B_s$, there is an edge $xy \in E(K')$ such that
  \[
    \sizeof{(N(x) \cap Z(K')) - y} \leq d_{s}(y) - d_{K'}(y).
    \]
    By the definition of $K'$, we have $\mu_K(xy) > 0$, that is, $xy \in E(K)$.
    Since $Z(K') = Z(K)$, it therefore suffices to show that $d_{s}(y) - d_{K'}(y) \leq d_{t}(y) - d_{K}(y)$. This follows from observing that if $J$ is the common underlying simple
    graph of $K$ and $K'$, then
    \[ d_{s}(y) - d_{K'}(y) = (s+1)[d_{B}(y) - d_{J}(y)] \leq (t+1)[d_{B}(y) - d_J(y)] = d_{t}(y) - d_{K}(y). \qedhere \]
\end{proof}
\begin{remark}
  The converse of Lemma~\ref{lem:flatten} is not true: for a simple graph $H$, it is possible that
  $\corefan(H_t) \leq t+1$ yet $\corefan(H) > 0$, where we consider $H$ itself as the graph $B_s$ for $s=0$.
  Consider the simple graph $H$ shown in Figure~\ref{fig:flattenconverse}.
  \begin{figure}
    \centering
    \begin{tikzpicture}
      \rpoint{w} (u1) at (0cm, 1cm) {};
      \apoint{} (u2) at (-1cm, 0cm) {};
      \apoint{} (u3) at (1cm, 0cm) {};
      \rpoint{u} (u4) at (0cm, -1cm) {};
      \rpoint{v} (v) at  (0cm, -2cm) {};
      \draw (u1) -- (u2) -- (u4) -- (u3) -- (u1); \draw (u2) -- (u3); \draw (v) -- (u4);
    \end{tikzpicture}
    \caption{Simple graph $H$ such that $\corefan(H_1) \leq 1$ but $\corefan(H) > 0$.}
    \label{fig:flattenconverse}
  \end{figure}

  To see that $\corefan(H) > 0$, consider the subgraph $K = H-v$. If $\corefan(H)\leq 0$, then there exists $xy\in E(K)$ with
  \[ \sum_{z \in Z}(d_K(z) - d_H(z) +1 - 0) \leq 1 \] for all
  $Z\subseteq N_K(x)$ with $y\in Z$. However, the only vertex in $K$
  that does not have the same degree in $K$ as it does in $H$ is $u$,
  and this difference in degree is only one. Hence $u$ is the only
  vertex that could contribute a nonpositive amount to this sum. If
  $x$ is not $u$ or $w$, then $x$ has degree $3$ in $K$, so the sum
  cannot be at most 1. On the other hand, if $x$ is $u$ or $w$, then
  while $x$ has only degree two in $K$, neither of these neighbours is
  $u$.

  To see that $\corefan(H_1) \leq 1$, let any full-multiplicity
  subgraph $K \subset H_1$ with $K\neq \emptyset$ be given. According
  to Lemma \ref{lem:tplus1} we need only find $xy \in E(K)$ with
  \[\sizeof{(N_{H_1}(x) \cap Z(K)) - y} \leq d_{H_1}(y) - d_K(y),\]
  where $Z(K) = \{v \in V(K) \st d_K(v) = d_H(v)\}$. If $uv \notin E(K)$ but $K$ does have at least one edge incident to $u$, then we may take $y=u$ and $x \in N_K(u)$, so that
  \[ \sizeof{(N_{H_1}(x) \cap Z(K)) - y} \leq 2 \leq d_{H_1}(y) -
    d_{K}(y). \] (Recall that $H_1$ has constant multiplicity $2$,
  so at a minimum, the two copies of $uv$ are missing at $y$ in the
  subgraph $K$.) Otherwise, we may choose $xy$ so that
  $(N_{H_1}(x) \cap Z(K)) - y =\emptyset$ and hence we immediately
  get our desired inequality: if $vu\in E(K)$ then we take
  $(x,y) = (v,u)$, and if $K$ has no edges incident to $u$ then
  $Z(K)$ is either $\emptyset$ or $\{w\}$, and in the latter
  case we may choose $y=w$.
\end{remark}

The following is the final result we need in order to write our proof of Theorem \ref{thm:Bqueuecorefan}.

\begin{theorem}\label{thm:qcorefan}
  Let $B$ be a simple graph. If $B$ has a full $B$-queue, then $\corefan(H) \leq 0$.
\end{theorem}

\begin{proof}
  Consider a full $B$-queue with vertex sequence $(u_1, \ldots, u_q)$ and set sequence $(S_1, \ldots, S_q)$.
  Let $K\subseteq H$ with $E(K)\neq\emptyset$. According to Lemma \ref{lem:tplus1} we need only find $xy \in E(K)$ with
    \[\sizeof{(N_{H}(x) \cap Z(K)) - y} \leq d_B(y) - d_K(y),\]
    where $Z(K) = \{v \in V(K) \st d_K(v) = d_H(v)\}$. We consider two cases:

    \caze{1}{$K$ contains an edge incident to $u_i$ for some $i$.}
    Choose $i$ to be the smallest such index, and let $x = u_i$. If
    there is some vertex in $N_K(u_i) \cap (S_i \setminus S_{i-1})$,
    let $y$ be such a vertex; otherwise, let $y$ be an arbitrary
    element of $N_K(x)$. We claim that
    $N_{H}(x) \cap Z(K) \subset \{y\}$. To this end, consider any
    $z \neq y \in N_K(x)$. The definition of a full $B$-queue implies
    that $z \in S_j$ for some $j \leq i$. Take the smallest such
    $j$. If $j=i$, then necessarily $z=y$, since
    $\sizeof{S_i \setminus (S_{i-1} \cup \{u_i\})} \leq 1$. Otherwise,
    $j < i$, and since $z\neq u_j$ (by choice of $i$) we know that
    $u_j \in N_B(z)$. But again, by our choice of $i$, the edge $u_jz$
    cannot be in $K$, and so $z \notin Z(K)$. It follows
    that \[ \sizeof{N_H(x) \cap Z(K) - y} = \sizeof{N_K(x) \cap Z(K) - y} = 0 \leq d_B(y) - d_K(y),\] as
    desired.

  \caze{2}{$K$ has no edges incident to $u_i$ for any $i$.} In this case, choose any $xy \in E(K)$. Since the $B$-queue is full, every vertex in $B$ is incident to at least one of $u_1, \ldots, u_q$, so $N(x) \cap Z(K) =\emptyset$. Thus again we have
  \[\sizeof{(N(x) \cap Z(K)) - y} = 0 \leq d_B(y) - d_K(y)\]
  as desired.
\end{proof}

\begin{remark} The converse of Theorem~\ref{thm:qcorefan} does not
  hold. In particular, there is an infinite family
  $\{H_p\}_{p \geq 4}$ of simple graphs such that for all $p$,
  $\corefan(H_p) = 0$ but $B=H_p$ does not have a full $B$-queue. To
  see this, let $B=H_p$ be the graph obtained from the complete graph
  $K_p$ by designating a special vertex $v$ and attaching $p-2$
  pendant edges to $v$. Let $z_1, \ldots, z_{p-2}$ be the vertices of
  degree $1$ adajcent to $v$. The graph $H_4$ is shown in
  Figure~\ref{fig:bqueueconverse}.
  \begin{figure}
    \centering
    \begin{tikzpicture}
      \apoint{} (w) at (0cm, 0cm) {};
      \rpoint{v} (u1) at (-90 : .67cm) {};
      \rpoint{} (u2) at (30 : 1cm) {};
      \rpoint{} (u3) at (150 : 1cm) {};
      \draw (u1) -- (u2) -- (u3) -- (u1);
      \foreach \i in {1,2,3} { \draw (w) -- (u\i); }
      \bpoint{z_1} (z1) at (-.5cm, -1.5cm) {};
      \bpoint{z_2} (z2) at (.5cm, -1.5cm) {};
      \draw (z2) -- (u1) -- (z1);
    \end{tikzpicture}
    \caption{The graph $H_4$, a simple graph with $\corefan(H_4) = 0$ that does not admit a full $B$-queue.}
    \label{fig:bqueueconverse}
  \end{figure}
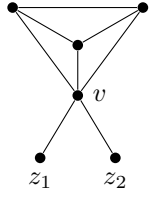
  If $H_p$ has a full $B$-queue with vertex sequence
  $(u_1, u_2, \ldots)$ then at least one vertex of the $K_p$ must
  occur as a $u_i$; choose the smallest such $i$. If $u_i \in S_{i-1}$,
  then $u_i=v$ and $\sizeof{S_i \setminus S_{i-1}} \geq \sizeof{N(u_i)\setminus S_{i-1}}\geq 3$, a
  contradiction. If $u_i\not\in S_{i-1}$, then $\sizeof{S_i\setminus S_{i-1}} \geq 3$,
  again a contradiction. Hence $H_p$ does not admit a full $B$-queue.

Now we claim that $\corefan(H_p) \leq 0$. We apply Lemma~\ref{lem:tplus1}.
  Let $K$ be any subgraph of $H_p$ with $E(K)\neq \emptyset$. If $K$ does not contain any of the pendant edges incident to $v$, but $v$ does have at least one incident edge in $K$,
  then let $y=v$ and take any $x \in N_K(v)$. Since $|N_K(x)|\leq p-1$ and $d_{H_p}(y) - d_K(y) \geq p-2$, we get that
  \[ \sizeof{(N_{H_p}(x) \cap Z(K))  - y} \leq p-2 \leq d_B(y) - d_K(y), \]
  as desired. Otherwise we can choose $xy\in E(K)$ so that $(N_{H_p}(x) \cap Z(K)) - y=\emptyset$ and hence we immediately get our desired inequality; if we can take $x=z_i$ for some $i$ then this is the case, else there are no edges incident to $v$. Since $v$ is joined to every vertex in $B$, this yields $Z(K) = \emptyset$, so that $N_{H_p}(x) \cap Z(K) = \emptyset$
  no matter which $x$ we choose.
\end{remark}

We can now prove Theorem \ref{thm:Bqueuecorefan}.

\begin{proof}[Proof of Theorem \ref{thm:Bqueuecorefan}] By Lemma \ref{lem:lowmult}, it suffices to show
  that $\corefan(H_{>t})\leq t$, where $H_{>t}$ is the subgraph of $H$
  consisting of the edges with multiplicity greater than $t$. If $H$
  has multiplicity at most $t$, then $H_{>t}$ is edgeless, and hence
  $\corefan(H_{>t})=0$ by definition. So, we may assume that $H$ has
  maximum multiplicity exactly $t+1$, and that $H_{>t}$ is the subgraph $H_t$ of $H$
  consisting precisely of all the edges in $H$ of multiplicity
  $t+1$. By Lemma \ref{lem:flatten}, we get our desired result of
  $\corefan(H_t)\leq t$ provided $\corefan(B)\leq 0$. Since $B$ has a
  full $B$-queue, Theorem \ref{thm:qcorefan} indeed tells us that
  $\corefan(B)\leq 0$, thus completing our proof.
\end{proof}

\section{Proof of Theorem \ref{thm:converse}}\label{sec:converse}

Before we begin the main proof of Theorem~\ref{thm:converse}, it will
help to record a lemma about constructing regular graphs with perfect
matchings.
\begin{lemma}\label{lem:regmatching}
  Let $n$, $k$, and $r$ be positive integers with $r\leq n$ and $k<n$. If $k$ and $r$ are even,
  then there is a $k$-regular simple graph $G$ on $n$ vertices containing a
  vertex set $S_r$ of size $r$ such that the induced subgraph $G[S_r]$
  has a perfect matching.
\end{lemma}
\begin{proof}
  For any even $k$ and any $n>k$, the standard circulant graph
  construction (see, e.g., Chapter~12 of \cite{handbook-algo}) yields
  a $k$-regular simple graph on $n$ vertices with a matching $M$ that covers
  at least $n-1$ vertices. In particular, $M$ covers $n$ vertices if
  $n$ is even, and $n-1$ vertices if $n$ is odd. Since $r$ is even, we
  see that the number of vertices covered by $M$ is at least
  $r$. Thus, one may choose any set of $r/2$ edges in $M$, and take
  $S_r$ to be the set of vertices covered by those edges.
\end{proof}

\begin{proof}[Proof of Theorem \ref{thm:converse}] Our goal is to build a graph $G$ whose $t$-core is $H$ and with $\fan(G)>\Delta(G)+t$ (and hence $\Fan(G)>\Delta+t$). Since $\corefan(H)>t$ there exists $K\subseteq H$, $E(K)\neq\emptyset$ with $\cdeg_{H,K}(x,y)> t$ for all $xy \in E(K)$, that is, with
\[\sum_{z \in Z}(d_K(z) - d_H(z) + \mu_K(x,z) - t) > 1\]
for some $Z\subseteq N_K(x)$ with $y\in Z$. Note that we may choose $K$ so that it has no isolated vertices. Choose positive integers $D$ and $r$ satisfying all of the following conditions:
  \begin{enumerate}[(a)]
  \item $r\geq \Delta(H) + 6 + t$
  \item $r$ is even
  \item $D \geq 3r + t$
  \item $D \geq \Delta(H) + 2r^2$,
  \item $D+t$ is an even multiple of $r-1$, with this multiple being at least 4.
  \end{enumerate}
  Initialize $G=H$. Our construction proceeds in several
  stages; at each stage, we will add vertices and/or edges to $G$. When the construction is complete we will verify that $G$ indeed has $t$-core $H$ and $\fan(G)>\Delta(G)+t$.

  \stage{1} In this stage, we will add vertices and edges to our initial $G=H$ in
  order to guarantee that $d_G(x) = D$ for all $x \in V(K)$.

  Let $p = \sizeof{V(K)}$, and write $V(K) = \{x_1, \ldots,
  x_p\}$. Note that since $K$ is not edgeless, we know that $p\geq
  2$. For each $x_i \in V(K)$, let $d_i = D - \deg_H(x_i)$.  We can
  write each $d_i$ as $d_i = \alpha_i (r-1) + \beta_i$ where $\alpha_i$
  and $\beta_i$ are integers with $0 \leq \beta_i \leq r-2$. We
  rewrite this equation as 
  \[ d_i = (\alpha_i - \beta_i)(r-1) + \beta_i r. \] Let
  $a_{i,r-1} = \alpha_i - \beta_i$ and let $a_{i,r} = \beta_i$. We know that $a_{i,r}\in\{0, \ldots, r-1\}$; note also that by assumption (d),
  \begin{align*}
  a_{i, r-1}&=\alpha_i-\beta_i
 = \left(\tfrac{d_i-\beta_i}{r-1}\right)-\beta_i
  =\left(\tfrac{D-\deg_H(x_i)-\beta_i}{r-1}\right)-\beta_i\\
  &\geq \left(\tfrac{D-\Delta(H)-\beta_i}{r-1}\right)-\beta_i \geq \left(\tfrac{2r^2-\beta_i}{r-1}\right)-\beta_i \geq r
  \end{align*}
  If $\sum a_{i,r}$ is odd, then we redefine the first pair
  $(a_{1,r}, a_{1,r-1})$ to be $(a_{1,r} +(r-1), a_{1,r-1} -r)$, which
  will change the parity of the sum since $r$ is even by assumption
  (b). Given the above inequality, we still have that
  $a_{i, r}, a_{i, r-1} \geq 0$ for all $i$; in fact we know that
  $a_{i, r-1}\geq r$ except possibly when $i=1$. We also still have that
  \[d_i=a_{i, r}(r)+ a_{i, r-1}(r-1).\]

  For $\ell \in \{r-1, r\}$, let $s_{\ell} = \sum_{i=1}^p a_{i,\ell}$,
  and let $S_{\ell}$ be a set of $s_{\ell}$ new vertices added to
  $G$. Our definition of the numbers $a_{i, \ell}$ guarantees that
  $s_r$ is even, and this will be helpful for us in a later stage of
  our construction.

  For each $x_i \in V(K)$, choose a disjoint set $T$ of $a_{i,\ell}$
  vertices from $S_{\ell}$, and add an edge of multiplicity $\ell$
  from $x_i$ to each vertex of $T$. Once we complete this procedure
  for all $i$ and both $\ell$, we see that every vertex in $K$ has
  degree $D$. Let $S = S_{r-1} \cup S_r$ with $s=|S|$.

  \stage{2} In this stage, we will add edges within $S$ to ensure that
  every vertex in $S_{\ell}$ ends with degree $D-{\ell}+t$.

  Our strategy in this stage will be, roughly, to first paste in a
  regular multigraph on the vertex set $S$ to bring the vertices of
  $S_{r-1}$ up to degree $D-(r-1)+1$ and the vertices of $S_r$ up to
  degree $D-r+t+2$, and then remove parallel copies of a (carefully
  planted) perfect matching from $S_r$ so that those vertices end with
  degree $D-r+t$.

  Let $k = \frac{D - 2(r-1)+t}{r-1} = \frac{D+t}{r-1} - 2$. By
  assumption (e), $k$ is an even integer and $k\geq 2$.  Since $sk$ is
  even, we can construct a $k$-regular simple graph on the vertex set
  $S$ provided $s>k$.  To verify this, start by observing the
  following, where we are using $r \geq \Delta(H)$ (by a weak version
  of assumption (a)):
  \[ s = \sum_{i=1}^p (a_{i-1} + a_i) \geq \sum_{i=1}^p \frac{d_i}{r} \geq \sum_{i=1}^p \frac{D-r}{r} \geq 2\frac{D-r}{r} = \frac{2D}{r} - 2. \]
  To show $s>k$, it remains to prove that $2D/r > (D+t)/(r-1)$, which is true iff $D>t\left(\tfrac{r}{r-2}\right)$. This follows from $D > 2t$ (by assumption (c) coupled with a weak version of assumption (a)) and $r \geq 4$ (by an even weaker version of assumption (a)).

  Let $A$ be a $k$-regular simple graph on the vertex set $S$. Since
  $s_r$ is even, as established in the previous stage of construction,
  Lemma~\ref{lem:regmatching} allows us to choose $A$ so that it
  contains a perfect matching $M$ on the vertex set $S_r$.

  We now modify $G$ as follows: make $G[S]$ have underlying graph $A$
  with edges in $A-M$ having multiplicity $r-1$ and edges in $M$
  having multiplicity $r-3$.

  Observe that at this point, every vertex in $S_{r-1}$ has degree
  \[ (r-1) + (r-1)k = (r-1) + (D+t - 2(r-1)) = D - (r-1)+t, \]
  while every vertex in $S_r$ has degree
  \[ r + (r-1)k - 2 = r + (D+t - 2(r-1)) - 2 = D-r+t. \]

  \stage{3} In this last stage, we will bring each vertex $v \in V(H) - V(K)$
  up to degree $D$. We do this by simply adding a single pendant edge
  of multiplicity $r-1$ to $v$, followed by enough pendant edges of
  multiplicity $1$ to obtain the desired degree. Note that this is possible since, by assumption (d),
  $D \geq \Delta(H) +2r^2 \geq \Delta(H)+ r-1$, so the pendant edge of multiplicity
  $r-1$ cannot itself make the degree greater than $D$.

This completes our construction of $G$.

  \textbf{Verification of Properties.} We begin by verifying that $H$ is the $t$-core of $G$. To this end, note that  $d(v) = D$ for
  all $v \in V(H)$. For all $v \in S$, we have
  $d(v) \leq D-(r-1)+t$, which is less than $D$ since $r-1 > t$ by a weak version of assumption (a). The endpoints of the pendant edges from Stage~3
  have $d(v) \leq r-1 < D$ (using assumption (e) weakly). Hence $\Delta(G) = D$, with this degree achieved precisely by the vertices in $H$. Now, note that every vertex
  of $H$ is incident to an edge of multiplicity $r-1$ or greater. So, for any vertex $v\in V(H)$,
  \[d(v) + \mu(v) \geq D + (r-1) > D+t.\]
  For $v \in S_r$ we have
  \[d(v) + \mu(v) = (D-r+t) + r = D+t,\]
  while for $v \in S_{r-1}$ we
  have
  \[d(v) + \mu(v) = (D - (r-1)+t) + (r-1) = D+t.\]
  The endpoints
  of the pendant edges added in Stage~3 have
  \[d(v) + \mu(v) \leq (r-1)+(r-1) \leq D+t,\]
  where the last inequality is another weak application of (e). Hence $H$ is indeed the $t$-core of $G$.

  To verify that $\fan(G) > D+t$, we choose $J=G[K\cup S]$  and show that
  $\deg_{J}(x,y) > D+t$ for all $xy \in E(J)$.

  We know that $\mu(x,y) \leq r$ for all $xy \in E(J)$, with this coming from a weakened assumption (a) (and the fact that $\Delta(H)\geq \mu(H)$) when $xy\in E(K)$. Using the computations above, we know that
  $d_J(v) \geq D-r$ for all $x \in V(J)$, with this coming from $r\geq \Delta(H)$ (again by assumption (a)) when $v\in V(K)$. We thus get
  \[ d_J(x) + d_J(y) - \mu(x,y) \geq 2D - 3r > D+t, \]
  where the last inequality is by assumption (c). Thus, Condition~(i) of the definition
  of $d_J(x,y)$ fails for $k=D+t$.

  Next we show that Condition~(ii) fails for $k=D+t$ as well. Consider
  any $xy \in E(J)$. We consider two cases,
  according to the location of $x$.

  \caze{1}{$x \in V(K)$.}  If $y \in V(K)$, let $y' = y$. Otherwise,
  since $x$ is not isolated in $K$, we can take some $y' \in N_K(x)$. Since
  $\corefan(K) > t$, there is a set $Z' \subset N_K(x)$ with
  $y' \in Z'$ such that
  \[ \sum_{z \in Z'}[ d_K(z) - d_H(z) + \mu_K(x,z) - t ] > 1. \]
  Since $J$ includes all the edges of $K$, we have $Z' \subset N_J(x)$, with
  $\mu_K(x,z) = \mu_J(x,z)$ for all $z \in Z'$. Furthermore, $d_H(z) - d_K(z) = D - d_J(z)$
  for all $z \in Z'$ since $d_G(z) = D$ and the only $G$-edges incident to $z$ not
  included in $J$ are the edges in $E(H) - E(K)$. Thus,
  \begin{equation}
    \label{ieq:xcase1}
    \sum_{z \in Z'}[ d_J(z) + \mu_J(x,z) - (D + t)] = \sum_{z \in Z'}[d_K(z) - d_H(z) + \mu_K(x,z) - t] > 1.
  \end{equation}
  If $y \in Z'$ then this immediately implies that $\deg_J(x,y) > D + t$. Otherwise,
  by our choice of $y'$, we have $y \in S_{\ell}$ for some $\ell$, in which case
  \[ d_J(y) + \mu(x,y) - (D+t) = (D-\ell+t) + \ell - (D+t) = 0, \] so letting
  $Z = Z' \cup \{y\}$ does not change the sum in
  Inequality~\eqref{ieq:xcase1}, and the set $Z$ witnesses
  $\deg_J(x,y) > D+t$.

  \greg{\caze{2}{$x \in S$.}} Let $y'$ be the unique neighbor of $x$ in $V(K)$.
  Observe that
  \begin{equation}
    \label{ieq:yprime}
     d_J(y') + \mu(x,y') - (D + t) \geq (D - \Delta(H)) + (r-1) - (D + t) \geq r-1-\Delta(H) - t \geq 5,
  \end{equation}
  where the second last inequality comes from assumption (a).

  Observe that for any $z \in N_J(x)$, even if $z \in S$ we still have
  \begin{equation}
    \label{ieq:minus3}
    d_J(z) + \mu(x,z) - (D+t) \geq (D-r+t) + (r-3) - (D+t) = -3.
  \end{equation}
  If $y = y'$, then taking any $z \in N_J(x) - y$ and putting $Z = \{z,y\}$, we see that Inequalities
  \eqref{ieq:yprime} and \eqref{ieq:minus3} together imply that $Z$ witnesses $\deg_J(x,y) > D+t$.
  (Such a $z$ exists because, by our construction, $d_J(x) \geq (r-1)k \geq 2\mu(x)$.)

  If $y \neq y'$, then since $y \in N_J(x)$, taking $Z = \{y, y'\}$
  again yields a set witnesing $\deg_J(x,y) > D+t$, via the same
  inequalities.
\end{proof}

\bibliographystyle{amsplain}
\bibliography{bib-bqueue}
\end{document}